\documentclass[12pt]{article}
\usepackage{amsmath,amsfonts,amssymb,amsthm}
\input amssym.def
\begin{document}
\def \Z{\mathbb Z}
\def \C{\mathbb C}
\def \R{\mathbb R}
\def \Q{\mathbb Q}
\def \N{\mathbb N}

\def \A{{\mathcal{A}}}
\def \D{{\mathcal{D}}}
\def \E{{\mathcal{E}}}
\def \E{{\mathcal{E}}}
\def \H{\mathcal{H}}
\def \S{{\mathcal{S}}}
\def \ZA{{\mathcal{Z}}}

\def \wt{{\rm wt}}
\def \tr{{\rm tr}}
\def \span{{\rm span}}
\def \Res{{\rm Res}}
\def \Der{{\rm Der}}
\def \End{{\rm End}}
\def \Ind {{\rm Ind}}
\def \Irr {{\rm Irr}}
\def \Aut{{\rm Aut}}
\def \GL{{\rm GL}}
\def \Hom{{\rm Hom}}
\def \mod{{\rm mod}}
\def \ann{{\rm Ann}}
\def \ad{{\rm ad}}
\def \rank{{\rm rank}\;}
\def \<{\langle}
\def \>{\rangle}

\def \g{{\frak{g}}}
\def \h{{\hbar}}
\def \k{{\frak{k}}}
\def \sl{{\frak{sl}}}
\def \gl{{\frak{gl}}}

\def \be{\begin{equation}\label}
\def \ee{\end{equation}}
\def \bex{\begin{example}\label}
\def \eex{\end{example}}
\def \bl{\begin{lem}\label}
\def \el{\end{lem}}
\def \bt{\begin{thm}\label}
\def \et{\end{thm}}
\def \bp{\begin{prop}\label}
\def \ep{\end{prop}}
\def \br{\begin{rem}\label}
\def \er{\end{rem}}
\def \bc{\begin{coro}\label}
\def \ec{\end{coro}}
\def \bd{\begin{de}\label}
\def \ed{\end{de}}

\newcommand{\m}{\bf m}
\newcommand{\n}{\bf n}
\newcommand{\nno}{\nonumber}
\newcommand{\nord}{\mbox{\scriptsize ${\circ\atop\circ}$}}
\newtheorem{thm}{Theorem}[section]
\newtheorem{prop}[thm]{Proposition}
\newtheorem{coro}[thm]{Corollary}
\newtheorem{conj}[thm]{Conjecture}
\newtheorem{example}[thm]{Example}
\newtheorem{lem}[thm]{Lemma}
\newtheorem{rem}[thm]{Remark}
\newtheorem{de}[thm]{Definition}
\newtheorem{hy}[thm]{Hypothesis}
\makeatletter \@addtoreset{equation}{section}
\def\theequation{\thesection.\arabic{equation}}
\makeatother \makeatletter

\begin{center}
{\Large \bf   $q$-Virasoro algebra  and affine Kac-Moody Lie algebras}
\end{center}
\begin{center}
{Hongyan Guo$^{a,c}$, Haisheng Li$^{a,b}$\footnote{Partially supported by
 China NSF grants (Nos.11471268, 11571391)},
Shaobin Tan$^{a}$\footnote{Partially supported by China NSF grants (Nos.11471268, 11531004)}
and Qing Wang$^{a}$\footnote{Partially supported by
 China NSF grants (Nos.11531004, 11622107), Natural Science Foundation of Fujian Province
(No. 2016J06002)}\\
$\mbox{}^{a}$School of Mathematical Sciences\\
 Xiamen University, Xiamen 361005, China\\
$\mbox{}^{b}$Department of Mathematical Sciences\\
Rutgers University, Camden, NJ 08102, USA\\
$\mbox{}^{c}$School of Mathematical and Statistical Sciences\\
University of Alberta, Edmonton T6G 2G1, Canada}
\end{center}

\begin{abstract}
We establish a natural connection of the $q$-Virasoro algebra $D_{q}$ introduced by Belov and Chaltikian
with affine Kac-Moody Lie algebras. More specifically,
for each abelian group $S$ together with a one-to-one linear character $\chi$,
we define an infinite-dimensional Lie algebra $D_{S}$ which reduces to $D_{q}$ when $S=\Z$.
Guided by the theory of equivariant quasi modules for vertex algebras,
 we introduce another Lie algebra $\g_{S}$ with $S$ as an automorphism group and
we prove that $D_{S}$ is isomorphic to the $S$-covariant algebra of the affine Lie algebra $\widehat{\g_{S}}$.
We then relate restricted $D_{S}$-modules of level $\ell\in \C$ to equivariant quasi modules
for the vertex algebra $V_{\widehat{\g_{S}}}(\ell,0)$ associated to $\widehat{\g_{S}}$ with level $\ell$.
Furthermore, we show that if $S$ is a finite abelian group of order $2l+1$, $D_{S}$ is
isomorphic to the affine Kac-Moody algebra of type $B^{(1)}_{l}$.
\end{abstract}

\section{Introduction}
In an attempt to find a meaningful $q$-analog of the Virasoro algebra, a certain infinite-dimensional Lie algebra  $D_{q}$ over the field $\C$ of complex numbers was introduced  and studied in \cite{BC} (cf. \cite{N}), where $q$ is a nonzero complex number.
In case that $q$ is not a root of unity,  Lie algebra $D_{q}$
was studied in \cite{gltw} by using vertex algebra theory in an essential way.
 It was proved therein that
the category of (suitably defined) restricted $D_{q}$-modules of a fixed level $\ell\in \C$ is canonically isomorphic
to the category of equivariant quasi modules for the vertex algebra $V_{\hat{\D}}(\ell,0)$
associated to the affine Lie algebra $\hat{\D}$  with level $\ell$ of a certain (infinite-dimensional) Lie algebra $\D$
with a symmetric invariant bilinear form.

In the current paper, we continue this study to deal with the case  where $q$ is a root of unity.
Actually, we work with a slightly more general setting.
Let $S$ be an abelian group with a one-to-one linear character
$\chi$. (This implies that if $S$ is finite, it must be cyclic.)
Then we define an infinite-dimensional Lie algebra $D_{S}$
with generators $D^{\alpha}(n)$ for $\alpha\in S,\  n\in \Z$, subject to certain relations.
In case $S=\Z$ with $q=\chi(1)$, $D_{S}$ reduces to $D_{q}$.
Guided by the theory of equivariant quasi modules for vertex algebras (see \cite{Li1, li-gamma-twisted}),
we introduce another Lie algebra $\g_{S}$ with $S$ as an automorphism group and
we prove that $D_{S}$ is isomorphic to the $S$-covariant algebra of the affine Lie algebra $\widehat{\g_{S}}$.
We then relate restricted $D_{S}$-modules of level $\ell\in \C$ to equivariant quasi modules
for the vertex algebra $V_{\widehat{\g_{S}}}(\ell,0)$ associated to $\widehat{\g_{S}}$ with level $\ell$.
Furthermore, we show that if $S$ is a finite abelian group of order $2l+1$, $D_{S}$ is
actually isomorphic to the affine Kac-Moody algebra of type $B^{(1)}_{l}$.

Now, we give a more detailed account of the contents of this paper.
Let $S$ be an abelian group with a one-to-one linear character
$\chi: S\rightarrow \C^{\times}$. We first define an infinite-dimensional Lie algebra $D_{S}$
with generators $D^{\alpha}(n)$ for $\alpha\in S,\  n\in \Z$, subject to certain relations
 (see Definition \ref{def-DS}).
%In case that $S=\Z$ with $q=\chi(1)$, $D_{S}$ reduces to $D_{q}$.
Following the common practice in vertex algebra theory, we form a generating function
$D^{\alpha}(x)= \sum_{n\in\Z}D^{\alpha}(n)x^{-n-1}$
for each $\alpha\in S$ and then calculate the commutator of two general generating functions.
The commutator relations for the generating functions and  theory of equivariant quasi modules for vertex algebras  indicate that $D_{S}$ is closely related to the affine Lie algebra of some (possibly infinite-dimensional) Lie algebra
with a symmetric invariant bilinear form.
Motivated by this,  for the abelian group $S$  we introduce another Lie algebra
$\g_{S}$ with generators $d^{\alpha,\beta}$ for $\alpha,\beta\in S$, subject to relations
$d^{-\alpha,\beta}=-d^{\alpha,\beta}$ and
\begin{eqnarray*}
&& [d^{\alpha,\beta},d^{\mu,\nu}] =
    \delta_{\alpha+\mu,\nu-\beta}d^{\alpha+\mu,-\alpha+\nu}
   -\delta_{\alpha+\mu,\beta -\nu}d^{\alpha+\mu,\alpha+\nu}
                         \nonumber\\
  &&\hspace{2.2cm}
   -\delta_{\alpha-\mu,\nu-\beta}d^{\alpha-\mu,\nu -\alpha}
   +\delta_{\alpha-\mu,\beta-\nu}d^{\alpha-\mu,\alpha+\nu}
\end{eqnarray*}
for $\alpha,\beta,\mu,\nu\in S$. It is proved that $S$ acts on $\g_{S}$ as an automorphism group with
$$\gamma\cdot d^{\alpha,\beta}=d^{\alpha,\beta+\gamma}\   \   \    \mbox{ for }\alpha,\beta,\gamma\in S.$$
We also equip $\g_{S}$ with a symmetric invariant bilinear form.
Then we prove that $D_{S}$ is isomorphic to the so-called $S$-covariant algebra of
the affine Lie algebra $\widehat{\g_S}$ as we explain next.

It is a classical fact that for any Lie algebra $K$ with an automorphism group $G$,
all the $G$-fixed points in $K$ form a Lie subalgebra $K^{G}$.  An important result in Kac-Moody algebra theory
 (see \cite{kac}) is that  twisted affine Kac-Moody algebras are isomorphic to
the fixed-point subalgebras of untwisted affine Kac-Moody algebras with respect to Dynkin diagram automorphisms.
Now, assume that  $G$ is a (possibly infinite) automorphism group of $K$ such that for any $a,b\in K$,
$$[ga,b]=0\    \    \mbox{ for all but finitely many }g\in G.$$
Under this condition, one can associate another Lie algebra to the pair $(K,G)$ as follows (see  \cite{G-K-K}, \cite{li-gamma-twisted}).
 Define a new multiplicative operation $[\cdot,\cdot]_{G}$  on $K$ by
 $$[a,b]_{G}=\sum_{g\in G}[ga,b].$$
 Set
 $$I_{G}={\rm span}\{ a-ga\ |\ a\in K,\ g\in G\}.$$
 Then $I_{G}$ is a two-sided ideal of the non-associative algebra $(K,[\cdot,\cdot]_{G})$ and the quotient algebra
 by $I_{G}$ is a Lie algebra, which is denoted by $K/G$ and called the {\em $G$-covariant algebra} of $K$.
 In case that $G$ is finite, one can show that Lie algebras $K^{G}$ and $K/G$ are isomorphic.

 Now that $D_{S}$ is isomorphic to the $S$-covariant algebra of the affine Lie algebra $\widehat{\g_S}$,
by applying a result of \cite{li-gamma-twisted} we show that
the category of restricted $D_{S}$-modules of a fixed level $\ell$ is naturally isomorphic to the category of equivariant quasi $V_{\hat{\g}_{S}}(\ell,0)$-modules. This generalizes the corresponding result of  \cite{gltw}.
Then  we continue to determine the Lie algebra $\g_S$ for $S$ a finite abelian group.
Assuming $|S|=2l+1$, we prove that $\g_S$ is a simple Lie algebra of type $B_{l}$.
It follows that $D_S$ is isomorphic to the $S$-covariant algebra of the affine Kac-Moody algebra
of type $B_{l}^{(1)}$. Using this and a result of Kac (see \cite{kac}) we show that $D_{S}$ is actually isomorphic to
the affine Kac-Moody algebra of type $B_{l}^{(1)}$.
For the case with $S$ of an even order, set $S^{0}=\{ \alpha\in S\ |\ 2\alpha=0\}$ and
let $I$ be the subspace of $\g_{S}$, linearly  spanned by the following vectors
$$d^{\alpha+\gamma,\beta+\gamma}-d^{\alpha,\beta}\  \  (\alpha,\beta\in S,\ \gamma \in S^0).$$
We prove that $I$ is an ideal of $\g_{S}$ and the quotient algebra $\g_S/I$ is
a direct sum of simple Lie algebras of type $B$ or a direct sum of simple Lie algebras of type
$D$ (see Proposition \ref{general-case}).

This paper is organized as follows: In Section 2, we introduce Lie algebras $D_{S}$ and $\g_S$, and prove that
$D_{S}$ is isomorphic to the $S$-covariant algebra of $\widehat{\g_S}$.
 In Section 3, we identify $\g_S$ as a subalgebra
of the general linear algebra $\gl_{S}$ and prove that  if $|S|=2l+1$, then $\g_S$ is isomorphic to $B_{l}$
and $D_{S}$ is isomorphic to the affine Kac-Moody algebra of type $B_{l}^{(1)}$.

\section{Lie algebras $D_{S}$, $\g_S$, and  equivariant quasi modules for $\Gamma$-vertex algebras}

In this section, to any abelian group $S$ with a linear character $\chi$ we associate two Lie algebras $D_S$ and $\g_S$.
Furthermore, we equip Lie algebra $\g_{S}$ with a symmetric invariant bilinear form
and show that the vertex algebra $V_{\widehat{\g_{S}}}(\ell,0)$ associated to the affine Lie algebra
$\widehat{\g_{S}}$ of level $\ell$ is a $\Gamma$-vertex algebra with $\Gamma=S$.
Then we establish an isomorphism between the category of restricted $D_{S}$-modules of level $\ell$
and that of equivariant quasi $V_{\widehat{\g_{S}}}(\ell,0)$-modules.

First of all, we fix some basic notations. Throughout this paper, %$\N$ denotes the set of nonnegative integers,
 $\C^{\times}$ denotes the multiplicative group of nonzero complex numbers
 (with $\C$ denoting the complex number field), and
the symbols $x,y,x_{1},x_{2},\dots $ denote mutually commuting independent formal variables. All vector spaces in this paper are
considered to be over  $\mathbb{C}$. For a vector space $U$, $U((x))$ is the vector space of lower
truncated integer power series in $x$ with coefficients
 in $U$, $U[[x]]$ is the vector space of nonnegative integer
 power series in $x$ with coefficients in $U$, and
$U[[x,x^{-1}]]$ is the vector space of doubly infinite integer
 power series in $x$ with coefficients in $U$.

We recall the $q$-analog of the Virasoro algebra from \cite{BC} (cf. \cite{N}).
Let $q$ be a nonzero complex number. By definition, $D_{q}$ is
the Lie algebra  with generators ${\bf c}$ and $D^{\alpha}(n)\ (\alpha,n\in\Z)$,
subjects to relations
$$D^{-\alpha}(n)=-D^{\alpha}(n),$$
 \begin{eqnarray}\label{eq:2.10}
 [D^{\alpha}(n),D^{\beta}(m)] & = &
    (q-q^{-1})[\alpha m-\beta n]_{q}D^{\alpha+\beta}(m+n)
                                                    \nonumber\\
    &&{} -(q-q^{-1})[\alpha m+\beta n]_{q}D^{\alpha-\beta}(m+n)
                                                               \nonumber\\
    &&{} +([m]_{q^{{\alpha}+{\beta}}}-[m]_{q^{{\alpha}-{\beta}}})\delta_{m+n,0}{\bf c}
\end{eqnarray}
for $m,n,\alpha,\beta\in\Z$,
where ${\bf c}$ is central and
$[n]_{q}$ is the $q$-integer defined by
        $$[n]_{q}= \frac{q^{n}-q^{-n}}{q-q^{-1}}.$$

\br{rlimit-cases}
{\em Here, it is understood that $[n]_{q}=n$
for $q= 1$ and  $[n]_{q}=n(-1)^{n-1}$ for $q=-1$. This is consistent with the fact that
$$\lim_{q\rightarrow  1}[n]_{q}=n\  \mbox{ and}\   \lim_{q\rightarrow  -1}[n]_{q}=n(-1)^{n-1}.$$
The following are some simple properties of the $q$-integers:
\begin{eqnarray}\label{eq:2.11}
[-n]_{q}=-[n]_{q}, \  \   \  \ [n]_{q^{-1}}=[n]_{q}, \  \   \  \  [m]_{q^{n}}=\frac{[mn]_{q}}{[n]_q}.
\end{eqnarray}}
\er

We generalize the Lie algebra $D_{q}$ as follows:

\bd{def-DS}
{\em Let $S$ be an abelian group with a group homomorphism $\chi: S\rightarrow \C^{\times}$, i.e., a linear character
of $S$.
Define $D_{S}$ to be the Lie algebra  with generators $D^{\alpha}(n)$ for $\alpha\in S,\ n\in \Z$,
subjects to the following relations
$$D^{-\alpha}(n)=-D^{\alpha}(n),$$
 \begin{eqnarray}\label{Sbracket}
 [D^{\alpha}(n),D^{\beta}(m)] & = &
    (\chi(m\alpha -n\beta)-\chi(n\beta-m\alpha))D^{\alpha+\beta}(m+n)
                                                    \nonumber\\
    &&{} -(\chi(m\alpha +n\beta)-\chi(-m\alpha-n\beta))D^{\alpha-\beta}(m+n)
                                                               \nonumber\\
    &&{} +\left([m]_{\chi(\alpha+\beta)}-[m]_{\chi(\alpha-\beta)}\right)\delta_{m+n,0}{\bf c}
\end{eqnarray}
for $m,n\in \Z,\ \alpha,\beta\in S$, where ${\bf c}$ is central.}
\ed

\br{nonroot-of-unity}
{\em Consider $S=\Z$. Let $q$ be a nonzero complex number and let $\chi$ be the linear character of $\Z$
given by $\chi(n)=q^{n}$ for $n\in \Z$. Then $D_{S}$ is reduced to the Lie algebra $D_{q}$.
In the case that $q$ is not a root of unity, Lie algebra $D_{q}$ was studied
in \cite{gltw} in the context of vertex algebras and their quasi modules.}
\er

%\br{root-of-unity}
%{\em Let $k$ be a positive integer and $q$ a primitive $k$-th root of unity. Set $S=\Z/k\Z$
%and let $\chi$ be the linear character given by $\chi(\alpha)=q^{\alpha}$ for $\alpha\in \Z/k\Z$,
%noticing that $q^{\alpha}$ and $[\alpha]_{q}$  are well defined for $\alpha\in \Z/k\Z$.
%In this case, $D_{S}$ is generated by
%$D^{\alpha}(n)$ for $\alpha\in\Z/k\Z,\ n\in\mathbb{Z}$.}
%\er

{\em From now on, we assume that $\chi: S\rightarrow \C^{\times}$ is one-to-one. }
Then, for $\alpha\in S$, $\chi(\alpha)=1$ if and only if $\alpha=0$.
For $\alpha\in S,$ set
\begin{eqnarray}\label{eq:2.16}
D^{\alpha}(x) = \sum_{n\in\mathbb{Z}}D^{\alpha}(n)x^{-n-1}.
\end{eqnarray}
Furthermore,  we set
 \begin{eqnarray}\label{eq:2.18}
\tilde{D}^{\alpha}(x)=D^{\alpha}(x)+ (1-\delta_{2\alpha,0})
\frac{1}{\chi(\alpha)-\chi(-\alpha)}{\bf c}x^{-1},
 \end{eqnarray}
where we use the one-to-one assumption on $\chi$.

By a straightforward calculation  (cf. \cite{gltw}) we have:

\begin{lem}\label{tildeD-characterization}
 The defining relations of $D_{S}$ are equivalent to
 \begin{eqnarray}\label{eq:2.19}
[{\bf c}, D_{S}]=0,\   \   \   \tilde{D}^{-\alpha}(x)=-\tilde{D}^{\alpha}(x),
 \end{eqnarray}
 \begin{eqnarray}\label{eq:2.20}
 [\tilde{D}^{\alpha}(x_{1}),\tilde{D}^{\beta}(x_{2})]% \nonumber\\
 &=&\chi(-\alpha)\tilde{D}^{\alpha+\beta}(\chi(-\alpha)x_{2})x_{1}^{-1}
         \delta\left(\frac{\chi(-\alpha-\beta)x_{2}}{x_{1}}\right)\nonumber\\
 &&-\chi(\alpha)\tilde{D}^{\alpha+\beta}(\chi(\alpha)x_{2})x_{1}^{-1}
            \delta\left(\frac{\chi(\alpha+\beta)x_{2}}{x_{1}}\right)
                                                            \nonumber\\
&&-\chi(-\alpha)\tilde{D}^{\alpha-\beta}(\chi(-\alpha)x_{2})x_{1}^{-1}
          \delta\left(\frac{\chi(\beta-\alpha)x_{2}}{x_{1}}\right) \nonumber\\
&&+\chi(\alpha)\tilde{D}^{\alpha-\beta}(\chi(\alpha)x_{2})x_{1}^{-1}
      \delta\left(\frac{\chi(\alpha-\beta)x_{2}}{x_{1}}\right)
                                                               \nonumber\\
&&-\chi(\alpha-\beta) \delta_{2(\alpha-\beta),0}\frac{\partial}{\partial x_{2}}
x_{1}^{-1}\delta\left(\frac{\chi(\alpha-\beta)x_{2}}{x_{1}}\right){\bf c}\nonumber\\
&&+\chi(\alpha+\beta)\delta_{2(\alpha+\beta),0}\frac{\partial}{\partial x_{2}}
x_{1}^{-1}\delta\left(\frac{\chi(\alpha+\beta)x_{2}}{x_{1}}\right){\bf c}\   \   \   \   \   \
 \end{eqnarray}
 for $\alpha,\beta\in S$.
 \end{lem}

For $\alpha,\beta\in S$, set
\begin{eqnarray}\label{eq:2.21}
D^{\alpha,\beta}(x) = \chi(\beta)\tilde{D}^{\alpha}(\chi(\beta)x).
\end{eqnarray}
From (\ref{eq:2.20}), we have
\begin{eqnarray}\label{eq:2.22}
&&[D^{\alpha,\beta}(x_{1}),D^{\mu,\nu}(x_{2})] \nonumber\\
&=& \chi(\beta+\nu)[\tilde{D}^{\alpha}(\chi(\beta)x_{1}),\tilde{D}^{\mu}(\chi(\nu)x_{2})] \nonumber\\
   & = &D^{\alpha+\mu,\nu-\alpha}(x_{2})x_{1}^{-1}
                \delta\left(\frac{\chi(\nu-\alpha)x_{2}}{\chi(\mu+\beta)x_{1}}\right)
              -D^{\alpha+\mu,\alpha+\nu}(x_{2})x_{1}^{-1}
                      \delta\left(\frac{\chi(\alpha+\nu)x_{2}}{\chi(\beta-\mu)x_{1}}\right)
                                                              \nonumber\\
         && -D^{\alpha-\mu,\nu-\alpha}(x_{2})x_{1}^{-1}
                 \delta\left(\frac{\chi(\nu-\alpha)x_{2}}{\chi(\beta-\mu)x_{1}}\right)
         +D^{\alpha-\mu,\alpha+\nu}(x_{2})x_{1}^{-1}
               \delta\left(\frac{\chi(\alpha+\nu)x_{2}}{\chi(\mu+\beta)x_{1}}\right)
                                                                   \nonumber\\
&& -\chi(\alpha-\mu)\delta_{2(\alpha-\mu),0}\frac{\partial}{\partial x_{2}}x_{1}^{-1}\delta\left(\frac{\chi(\alpha-\mu+\nu-\beta)x_{2}}{x_{1}}\right){\bf c}\nonumber\\
   &&    +\chi(\alpha+\mu)\delta_{2(\alpha+\mu),0}
       \frac{\partial}{\partial x_{2}}x_{1}^{-1}\delta\left(\frac{\chi(\alpha+\mu+\nu-\beta)x_{2}}{x_{1}}\right){\bf c}\  \  \  \
\end{eqnarray}
for $\alpha,\beta,\mu,\nu\in S$.

Using this we get another characterization of the Lie algebra $D_{S}$.

\begin{prop}\label{psecond-def-D}
Lie algebra $D_{S}$ is isomorphic to the Lie algebra ${\mathcal{L}}$ with generators ${\bf c}$
and $D^{\alpha,\beta}(n)$
for $\alpha,\beta\in S,\  n\in \Z$, subject to relations
\begin{eqnarray}\label{eq:2.23}
[{\bf c}, {\mathcal{L}}]=0,\   \   \   D^{-\alpha,\beta}(n)=-D^{\alpha,\beta}(n),\  \  \
D^{\alpha,\beta+\gamma}(n)=\chi(\gamma)^{-n}D^{\alpha,\beta}(n)
\end{eqnarray}
 and the relation (\ref{eq:2.22})
 with $D^{\alpha,\beta}(x)=\sum\limits_{n\in\Z}D^{\alpha,\beta}(n)x^{-n-1}$.
\end{prop}

\begin{proof} Given Lie algebra $D_{S}$, with $D^{\alpha,\beta}(x)$ defined in (\ref{eq:2.21}),
we see that (\ref{eq:2.23}) and (\ref{eq:2.22}) hold.
It follows that there is a Lie algebra homomorphism $\theta$ from $\mathcal{L}$ onto $D_{S}$
sending $D^{\alpha,\beta}(x)$ to $\chi(\beta)\tilde{D}^{\alpha}(\chi(\beta)x)$
for $\alpha,\beta\in S$.
On the other hand, for Lie algebra $\mathcal{L}$ we have
$$D^{\alpha,\beta}(x)=\chi(\beta)D^{\alpha,0}(\chi(\beta)x)\  \ \mbox{ and }\  D^{-\alpha,0}(x)=-D^{\alpha,0}(x)
\ \mbox{ for }\ \alpha,\beta\in S.$$
It is clear that relation (\ref{eq:2.22}) is equivalent to (\ref{eq:2.20})
with $\tilde{D}^{\alpha}(x)=D^{\alpha,0}(x)$ for $\alpha\in S$.
Then there exists a natural Lie algebra homomorphism from $D_{S}$ onto ${\mathcal{L}}$,
sending $\tilde{D}^{\alpha}(x)$ to $D^{\alpha,0}(x)$
for $\alpha\in S$. Consequently, $\theta$ is a Lie algebra isomorphism.
\end{proof}

%The defining relations of $D_{S}$ are also equivalent to
%\begin{eqnarray}\label{eqDS}
%D^{-\alpha,\beta}(x)=-D^{\alpha,\beta}(x),\   \  \  \   D^{\alpha,\beta+\gamma}(x)=\chi(\gamma)D^{\alpha,\beta}(\chi(\gamma)x),
%\end{eqnarray}
%\begin{eqnarray}\label{eqS}
%&&[D^{\alpha,\beta}(x_{1}),D^{\mu,\nu}(x_{2})] \nonumber\\
 %  & = &D^{\alpha+\mu,\nu-\alpha}(x_{2})x_{1}^{-1}
%                \delta\left(\frac{\chi(\nu-\alpha)x_{2}}{\chi(\mu+\beta)x_{1}}\right)
%              -D^{\alpha+\mu,\alpha+\nu}(x_{2})x_{1}^{-1}
%                      \delta\left(\frac{\chi(\alpha+\nu)x_{2}}{\chi(\beta-\mu)x_{1}}\right)\nonumber\\
 %        && -D^{\alpha-\mu,s-\alpha}(x_{2})x_{1}^{-1}
%                 \delta\left(\frac{\chi(\nu-\alpha)x_{2}}{\chi(\beta-\mu)x_{1}}\right)
 %        +D^{\alpha-\mu,\alpha+\nu}(x_{2})x_{1}^{-1}
 %              \delta\left(\frac{\chi(\alpha+\nu)x_{2}}{\chi(\mu+\beta)x_{1}}\right) \nonumber\\
%&& -\chi(\alpha-\mu)\delta_{2(\alpha-\mu),0}\frac{\partial}{\partial x_{2}}
%x_{1}^{-1}\delta\left(\frac{\chi(\alpha-\mu+\nu-\beta)x_{2}}{x_{1}}\right){\bf c}\nonumber\\
 %  &&    +\chi(\alpha+\mu)\delta_{2(\alpha+\mu),0}
 %      \frac{\partial}{\partial x_{2}}x_{1}^{-1}\delta\left(\frac{\chi(\alpha+\mu+\nu-\beta)x_{2}}{x_{1}}\right){\bf c}\  \  \  \
%\end{eqnarray}
%for $\alpha,\beta,\mu,\nu\in S$.

\br{rS0}
{\em Set
\begin{eqnarray}
S^{0}=\{ \alpha\in S\ |\ 2\alpha =0\},
\end{eqnarray}
a subgroup of $S$.  Note that if $S^{0}=\{0\}$, relation (\ref{eq:2.22}) becomes
\begin{eqnarray}\label{eqS-odd}
&&[D^{\alpha,\beta}(x_{1}),D^{\mu,\nu}(x_{2})] \nonumber\\
   & = &D^{\alpha+\mu,\nu-\alpha}(x_{2})x_{1}^{-1}
                \delta\left(\frac{\chi(\nu-\alpha)x_{2}}{\chi(\mu+\beta)x_{1}}\right)
              -D^{\alpha+\mu,\alpha+\nu}(x_{2})x_{1}^{-1}
                      \delta\left(\frac{\chi(\alpha+\nu)x_{2}}{\chi(\beta-\mu)x_{1}}\right)
                                                              \nonumber\\
         && -D^{\alpha-\mu,\nu-\alpha}(x_{2})x_{1}^{-1}
                 \delta\left(\frac{\chi(\nu-\alpha)x_{2}}{\chi(\beta-\mu)x_{1}}\right)
         +D^{\alpha-\mu,\alpha+\nu}(x_{2})x_{1}^{-1}
               \delta\left(\frac{\chi(\alpha+\nu)x_{2}}{\chi(\mu+\beta)x_{1}}\right)
                                                                   \nonumber\\
&& +\left(\delta_{\alpha,-\mu}-\delta_{\alpha,\mu}\right)\frac{\partial}{\partial x_{2}}
x_{1}^{-1}\delta\left(\frac{\chi(\nu-\beta)x_{2}}{x_{1}}\right){\bf c}\  \  \  \
\end{eqnarray}
for $\alpha,\beta,\mu,\nu\in S$. }
\er

In the following, we shall relate $D_{S}$ with the affine algebra of a new Lie algebra.

\bd{gS} {\em Let $S$ be an abelian group as before.
Define $\g_{S}$ to be the Lie algebra with generators $d^{\alpha,\beta}$
for $\alpha, \beta\in S$,
subject to relations
\begin{eqnarray}\label{ed-symmetry}
d^{-\alpha,\beta}=-d^{\alpha,\beta},
\end{eqnarray}
\begin{eqnarray}
&& [d^{\alpha,\beta},d^{\mu,\nu}] =
    \delta_{\alpha+\mu,\nu-\beta}d^{\alpha+\mu,-\alpha+\nu}
   -\delta_{\alpha+\mu,\beta -\nu}d^{\alpha+\mu,\alpha+\nu}
                         \nonumber\\
  &&\hspace{2.2cm}
   -\delta_{\alpha-\mu,\nu-\beta}d^{\alpha-\mu,\nu -\alpha}
   +\delta_{\alpha-\mu,\beta-\nu}d^{\alpha-\mu,\alpha+\nu}  \label{eq:2.24}
\end{eqnarray}
for $\alpha,\beta,\mu,\nu\in S$.}
\ed

The following result can be found in \cite{li-gamma-twisted} (Lemma 4.1) (cf. \cite{G-K-K}):

\begin{lem}\label{lcovariant-Lie-algebra}
Let $K$ be any Lie algebra with a symmetric invariant bilinear form $\<\cdot,\cdot\>$ and
let $G$ be an automorphism group of $K$, preserving $\<\cdot,\cdot\>$,
such that for any $a,b\in K$,
$$[ga,b]=0\    \mbox{ and }\  \<ga,b\>=0\   \   \  \mbox{ for all but finitely many }g\in G.$$
 Define a new multiplicative operation $[\cdot,\cdot]_{G}$ and a new bilinear form $\<\cdot,\cdot\>_{G}$ on $K$ by
 $$[a,b]_{G}=\sum_{g\in G}[ga,b]\     \mbox{ and }\   \<a,b\>_{G}=\sum_{g\in G}\<ga,b\>.$$
 Set
 $$I_{G}={\rm span}\{ a-ga\ |\ a\in K,\ g\in G\}.$$
 Then $I_{G}$ is a two-sided ideal of the non-associative algebra $(K,[\cdot,\cdot]_{G})$ and the quotient algebra
 modulo $I_{G}$ is a Lie algebra, which is denoted by $K/G$ and called the {\em $G$-covariant algebra} of $K$.
 Furthermore, $\<\cdot,\cdot\>_{G}$ reduces to a symmetric invariant bilinear form on $K/G$.
 \end{lem}

The following is a more explicit construction of the Lie algebra $\g_{S}$:

\bp{Kalgebra}
Let $K$ be a vector space with a basis $\{ F_{\alpha,\beta}\ |\ \alpha,\beta\in S\}$.
Define a multiplicative operation $*$ on $K$ by
\begin{eqnarray*}
F_{\alpha,\beta}*F_{\mu,\nu} &=&
    \delta_{\alpha+\mu,\nu-\beta}F_{\alpha+\mu,-\alpha+\nu}
   -\delta_{\alpha+\mu,\beta -\nu}F_{\alpha+\mu,\alpha+\nu}\nonumber\\
            &&   -\delta_{\alpha-\mu,\nu-\beta}F_{\alpha-\mu,\nu -\alpha}
   +\delta_{\alpha-\mu,\beta-\nu}F_{\alpha-\mu,\alpha+\nu}
\end{eqnarray*}
for $\alpha,\beta,\mu,\nu\in S$. Set
$$J={\rm span}\{ F_{-\alpha,\beta}+F_{\alpha,\beta}\ |\  \alpha,\beta\in S\}.$$
Then $J$ is a two-sided ideal of the non-associative algebra $(K,*)$ and  the quotient algebra $K/J$ is a Lie algebra.
Furthermore, $\g_{S}$ is isomorphic to the Lie algebra $K/J$ with $d^{\alpha,\beta}$ corresponding to
$F_{\alpha,\beta}+J$ for $\alpha,\beta\in S$.
\ep

\begin{proof} We shall apply Lemma \ref{lcovariant-Lie-algebra}.
First, define an operation ``$\cdot$'' on $K$ by
\begin{eqnarray}
F_{\alpha,\beta}\cdot F_{\mu,\nu} =\delta_{\nu, \alpha+\beta+\mu}F_{\alpha+\mu,\beta+\mu}
\end{eqnarray}
for $\alpha,\beta,\mu,\nu\in S$.
It is straightforward to show that $(K,\cdot)$ is an associative algebra. Then $K$ is naturally a Lie algebra, where
\begin{eqnarray*}
[F_{\alpha,\beta},F_{\mu,\nu}]
&=&\delta_{\nu, \alpha+\beta+\mu}F_{\alpha+\mu,\beta+\mu}-\delta_{\beta, \mu+\nu+\alpha}F_{\mu +\alpha,\nu +\alpha}\\
&=& \delta_{\alpha+\mu,\nu-\beta}F_{\alpha+\mu,-\alpha+\nu}
   -\delta_{\alpha+\mu,\beta -\nu}F_{\alpha+\mu,\alpha+\nu}
\end{eqnarray*}
for $\alpha,\beta,\mu,\nu\in S$.
Let $\theta$ be the linear endomorphism of $K$ defined by
$$\theta(F_{\alpha,\beta})=F_{-\alpha,\beta}\    \    \   \mbox{ for }\alpha,\beta \in S.$$
It is straightforward to show that $\theta$ is an order $2$ anti-automorphism of the associative algebra $K$.
Then $-\theta$ is an order $2$ automorphism of the Lie algebra $K$.   From  Lemma \ref{lcovariant-Lie-algebra}, we have
\begin{eqnarray*}
[F_{\alpha,\beta},F_{\mu,\nu}]_{\<-\theta\>} &=&
    \delta_{\alpha+\mu,\nu-\beta}F_{\alpha+\mu,-\alpha+\nu}
   -\delta_{\alpha+\mu,\beta -\nu}F_{\alpha+\mu,\alpha+\nu}\nonumber\\
 &&  -\delta_{-\alpha+\mu,\nu-\beta}F_{-\alpha+\mu,\alpha+\nu}+\delta_{-\alpha+\mu,\beta-\nu}F_{-\alpha+\mu,\nu -\alpha}
\end{eqnarray*}
for $\alpha,\beta,\mu,\nu\in S$. Then the first part of this lemma follows immediately. As for the second part, we first have a
Lie algebra homomorphism $\psi$ from $\g_{S}$ onto $K/J$
with $\psi(d^{\alpha,\beta})=F_{\alpha,\beta}+J$ for $\alpha,\beta\in S$.
By using a basis of $K/J$, we see that $\psi$ is actually an isomorphism.
\end{proof}

From Proposition \ref{Kalgebra} we readily have:

\begin{lem}\label{GSautomorphism}
The abelian group $S$ acts on $\g_{S}$ as an automorphism group with
\begin{eqnarray}
\gamma \cdot d^{\alpha,\beta}=d^{\alpha,\beta+\gamma}
\end{eqnarray}
for $\gamma,\alpha,\beta\in S$.
\end{lem}

Now,  let $\chi: S\rightarrow \C^{\times}$ be a linear character as in the definition of $D_{S}$.
It is straightforward to see that $\g_{S}$ admits a bilinear form $\<\cdot,\cdot\>_{\chi}$ such that
\begin{eqnarray}\label{eform-gS}
\langle d^{\alpha,\beta},d^{\mu,\nu}\rangle_{\chi}=
\chi(\alpha+\mu)\delta_{2(\alpha+\mu),0}\delta_{\alpha+\mu,\beta-\nu}
-\chi(\alpha-\mu)\delta_{2(\alpha-\mu),0}\delta_{\alpha-\mu,\beta-\nu}
\end{eqnarray}
for $\alpha,\beta,\mu,\nu\in S$. We have:

\begin{lem}
The defined bilinear form $\langle\cdot,\cdot\rangle_{\chi}$ on Lie algebra $\g_{S}$ is symmetric and invariant, and
 the group action of $S$ preserves $\langle\cdot,\cdot\rangle_{\chi}$.  Furthermore, for any $\alpha,\beta,\mu,\nu\in S$,
\begin{eqnarray}
[d^{\alpha,\beta+\gamma},d^{\mu,\nu}]=0\   \   \mbox{ and } \  \
\langle d^{\alpha,\beta+\gamma},d^{\mu,\nu}\rangle_{\chi}=0
\end{eqnarray}
for all but finitely many $\gamma\in S$.
\end{lem}

\begin{proof} We shall use Lemma \ref{lcovariant-Lie-algebra}.
First, recall the associative algebra $K$ from the proof of Proposition \ref{Kalgebra}, where $K$ has a basis
$\{ F_{\alpha,\beta}\ | \  \alpha,\beta\in S\}$ with
\begin{eqnarray*}
F_{\alpha,\beta}\cdot F_{\mu,\nu} =\delta_{\nu, \alpha+\beta+\mu}F_{\alpha+\mu,\beta+\mu}
\end{eqnarray*}
for $\alpha,\beta,\mu,\nu\in S$.
We define a bilinear form $\<\cdot,\cdot\>$ on $K$ by
\begin{eqnarray*}
\langle F_{\alpha,\beta},F_{\mu,\nu}\rangle=
\chi(\alpha+\mu)\delta_{2(\alpha+\mu),0}\delta_{\alpha+\mu,\beta-\nu}
\end{eqnarray*}
for $\alpha,\beta,\mu,\nu\in S$. It can be readily seen that $\<\cdot,\cdot\>$ is symmetric.
On the other hand, $\<\cdot,\cdot\>$ is also associative because
$$\langle F_{\alpha,\beta},F_{\mu,\nu}\rangle=f\left(F_{\alpha,\beta}\cdot F_{\mu,\nu}\right)$$
for $\alpha,\beta,\mu,\nu\in S$, where $f$ is the linear functional on $K$ given by
$$f(F_{\alpha,\beta})=\chi(\alpha)\delta_{2\alpha,0}\   \   \   \mbox{ for }\alpha,\beta \in S.$$
Consequently, $\<\cdot,\cdot\>$ is a
symmetric and invariant bilinear form on $K$ viewed as a Lie algebra.
It is also clear that the Lie algebra automorphism $-\theta$ of $K$ preserves $\<\cdot,\cdot\>$.
For $\alpha,\beta,\mu,\nu\in S$, we have
\begin{eqnarray*}
\langle F_{\alpha,\beta},F_{\mu,\nu}\rangle_{\<-\theta\>}
&=&\<F_{\alpha,\beta},F_{\mu,\nu}\>- \<F_{-\alpha,\beta},F_{\mu,\nu}\> \\
&=&
\chi(\alpha+\mu)\delta_{2(\alpha+\mu),0}\delta_{\alpha+\mu,\beta-\nu}
-\chi(\alpha-\mu)\delta_{2(\alpha-\mu),0}\delta_{\alpha-\mu,\beta-\nu}.
\end{eqnarray*}
Then the first part of this lemma follows immediately from Lemma \ref{lcovariant-Lie-algebra}. The second part is clear.
\end{proof}

\br{rS0-form}
{\em Note that in case that $S^{0}=\{0\}$ (where $S^0=\{ \alpha\in S\ |\ 2\alpha=0\}$), we have
\begin{eqnarray}\label{bilinear-form-gs}
\langle d^{\alpha,\beta},d^{\mu,\nu}\rangle_{\chi}=
\left(\delta_{\alpha,-\mu}-\delta_{\alpha,\mu}\right)\delta_{\beta,\nu}
\end{eqnarray}
for $\alpha,\beta,\mu,\nu\in S$, recalling (\ref{eform-gS}). }
\er

Let $(S,\chi)$ be given as in the definition of $D_{S}$. We have a Lie algebra $\g_{S}$
with a symmetric invariant bilinear form $\<\cdot,\cdot\>_{\chi}$.
Associated to the pair $(\g_{S},\<\cdot,\cdot\>_{\chi})$, we have an affine Lie algebra
$$\widehat{\g_{S}}=\g_{S}\otimes \C[t,t^{-1}]\oplus \C {\bf k}.$$
Naturally, $S$ acts on $\widehat{\g_{S}}$ as an automorphism group, where
\begin{eqnarray}
\gamma\cdot {\bf k}={\bf k},\    \    \    \
\gamma\cdot (d^{\alpha,\beta}(n))=\chi(\gamma)^{n}d^{\alpha,\beta+\gamma}(n)
\end{eqnarray}
for $\gamma,\alpha,\beta\in S,\ n\in \Z$.

Now, we are in a position to present our first main result.

\bt{pisomorphism}
The linear map $\pi:  \widehat{\g_{S}}\rightarrow D_{S}$, defined by $\pi ({\bf k})={\bf c}$ and
 $$\pi (d^{\alpha,\beta}(x))=D^{\alpha,\beta}(x)=\chi(\beta)\tilde{D}^{\alpha}(\chi(\beta)x)\   \   \mbox{ for }\alpha,\beta\in S,$$
gives rise to a Lie algebra  isomorphism from the covariant algebra $\widehat{\g_{S}}/S$ to $D_{S}$.
\et

\begin{proof} Recall the defining relations (\ref{eq:2.19}) and (\ref{eq:2.20}) of $D_{S}$.
As the relations $d^{-\alpha,\beta}=-d^{\alpha,\beta}$  for $\alpha,\beta\in S$ hold in $\g_{S}$, the linear map $\pi$ is well defined.
On the other hand,
note that $\widehat{\g_{S}}/S$ as a vector space  is the quotient of $\widehat{\g_{S}}$ by
the relations
$$d^{\alpha,\beta+\gamma}(x)=\chi(\gamma)d^{\alpha,\beta}(\chi(\gamma)x)\   \    \   \mbox{ for }\alpha,\beta,\gamma\in S.$$
Then  $\pi$ is a linear isomorphism.
 Let $\alpha,\beta,\mu,\nu\in S,\ m,n\in \Z$. We have
\begin{eqnarray*}
&&[d^{\alpha,\beta}(m),d^{\mu,\nu}(n)]_{S}\nonumber\\
&=&\sum_{\gamma\in S}\chi(\gamma)^{m}[d^{\alpha,\beta+\gamma}(m),d^{\mu,\nu}(n)]\nonumber\\
&=&\sum_{\gamma\in S}\chi(\gamma)^{m} (\delta_{\alpha+\mu,\nu-\beta-\gamma}d^{\alpha+\mu,-\alpha+\nu}(m+n)
   -\delta_{\alpha+\mu,\beta+\gamma -\nu}d^{\alpha+\mu,\alpha+\nu}(m+n)
                         \nonumber\\
  &&\  \  \   \   -\delta_{\alpha-\mu,\nu-\beta-\gamma}d^{\alpha-\mu,\nu -\alpha}(m+n)
   +\delta_{\alpha-\mu,\beta+\gamma-\nu}d^{\alpha-\mu,\alpha+\nu}(m+n))\nonumber\\
   &&+m\delta_{m+n,0}{\bf k}\sum_{\gamma\in S}\chi(\gamma)^{m}\left(\chi(\alpha+\mu)\delta_{2(\alpha+\mu),0}\delta_{\alpha+\mu,\beta+\gamma-\nu}
-\chi(\alpha-\mu)\delta_{2(\alpha-\mu),0}\delta_{\alpha-\mu,\beta+\gamma-\nu}\right)\nonumber\\
&=&   \chi(\nu-\alpha-\beta-\mu)^{m}d^{\alpha+\mu,-\alpha+\nu}(m+n)
   -\chi(\alpha+\mu+\nu-\beta)^{m}d^{\alpha+\mu,\alpha+\nu}(m+n)
                         \nonumber\\
  &&-\chi(\mu+\nu-\alpha-\beta)^{m}d^{\alpha-\mu,\nu -\alpha}(m+n)
   +\chi(\alpha+\nu-\beta-\mu)^{m}d^{\alpha-\mu,\alpha+\nu}(m+n)\nonumber\\
   &&+m\delta_{m+n,0}{\bf k}\left(\chi(\alpha+\mu)\delta_{2(\alpha+\mu),0}\chi(\alpha+\mu+\nu-\beta)^{m}
-\chi(\alpha-\mu)\delta_{2(\alpha-\mu),0}\chi(\alpha+\nu-\beta-\mu)^{m}\right).
\end{eqnarray*}
Comparing this with (\ref{eq:2.22}), we see that $\pi$ is a Lie algebra homomorphism, and
hence it is an isomorphism.
\end{proof}

%We recall the definition of a vertex algebra (cf. \cite{LL}).
%A {\em vertex algebra} is a vector space $V$ equipped with a linear map
             %\begin{eqnarray*}
            % Y(\cdot,x):&&V\longrightarrow \mathrm{Hom}(V,V((x)))\subset \mathrm{EndV}[[x,x^{-1}]]\\
             % &&v\longmapsto Y(v,x)=\sum_{n\in\mathbb{Z}}v_{n}x^{-n-1}\ \ (\mbox{where }v_{n}\in \End V)
              %\end{eqnarray*}
              %and with a distinguished vector $\textbf{1}\in V$, called the {\em vacuum vector},
              %such that all the following conditions are satisfied for $u,v\in V$:
             % $$Y(\textbf{1},x)v=v,$$
             % $$Y(v,x)\textbf{1}\in V[[x]] \;\;\mbox{and}\;\; \lim_{x\mapsto 0}Y(v,x)\textbf{1} = v,$$
              %and
              %\begin{eqnarray*}
            %  &&x_{0}^{-1}\delta\left(\frac{x_{1}-x_{2}}{x_{0}}\right)Y(u,x_{1})Y(v,x_{2}) -
                    %     x_{0}^{-1}\delta\left(\frac{x_{2}-x_{1}}{-x_{0}}\right)Y(v,x_{2})Y(u,x_{1})\\
              % &&\hspace{2cm}= x_{2}^{-1}\delta\left(\frac{x_{1}-x_{0}}{x_{2}}\right)Y(Y(u,x_{0})v,x_{2})
            %  \end{eqnarray*}
            %  (the {\em Jacobi identity}).

Recall the following notion of quasi module for a vertex algebra from  \cite{Li1}:

\bd{quasi-module}
{\em Let $V$ be a vertex algebra. A {\em quasi $V$-module} is a vector space
$W$ equipped with a linear map
\begin{eqnarray*}
Y_{W}(\cdot,x): & V\rightarrow & \mathrm{Hom}(W,W((x)))\subset (\End W)[[x,x^{-1}]],\\
&v\mapsto & Y_{W}(v,x)
\end{eqnarray*}
satisfying the conditions that
                   $$ Y_{W}({\bf 1},x) = {\bf 1}_{W}\ \ (\mbox{the identity operator on }W)$$
                    and  that for any $u,v\in V$, there exists a nonzero polynomial $p(x_{1},x_{2})$
                    such that
 \begin{eqnarray}
&&{}x_{0}^{-1}\delta\left(\frac{x_{1}-x_{2}}{x_{0}}\right)p(x_{1},x_{2})Y_{W}(u,x_{1})Y_{W}(v,x_{2})
                    \nonumber\\
&&\hspace{1cm} -x_{0}^{-1}\delta\left(\frac{x_{2}-x_{1}}{-x_{0}}\right)p(x_{1},x_{2})Y_{W}(v,x_{2})Y_{W}(u,x_{1})
                                                               \nonumber\\
 &= & x_{2}^{-1}\delta\left(\frac{x_{1}-x_{0}}{x_{2}}\right)p(x_{1},x_{2})Y_{W}(Y(u,x_{0})v,x_{2}). \nonumber
\end{eqnarray}}
\ed

We also recall the definitions of a $\Gamma$-vertex algebra and an equivariant quasi module from \cite{li-gamma-twisted}.

\bd{def.2.6}
{\em Let $\Gamma$ be a group.
A {\em $\Gamma$-vertex algebra} is a vertex algebra $V$ equipped with group homomorphisms
                $$ R: \Gamma \rightarrow {\mathrm{GL}}(V); \ g \longmapsto R_{g}, \  \mbox{and}\
                 \chi: \Gamma \rightarrow \mathbb{C}^{\times}$$
                 such that $R_{g}(\textbf{1}) = \textbf{1}$ for $g\in\Gamma$
                 and
                 $$R_{g}Y(v,x)R_{g}^{-1} = Y(R_{g}(v), \chi(g)^{-1}x) \;\;\mbox{for}\; g\in\Gamma, \  v\in V.$$}
\ed

Note that this is a reformulation of the notion of $\Gamma$-vertex algebra defined in \cite{Li1}.

\bd{equivariant-module}
{\em Let $V$ be a $\Gamma$-vertex algebra. An {\em equivariant quasi $V$-module} is a quasi module $(W, Y_{W})$
                for $V$ viewed as a vertex algebra, satisfying the condition that
                $$Y_{W}(R_{g}(v),x) = Y_{W}(v, \chi(g)x) \;\;\mbox{for}\; g\in\Gamma,\  v\in V$$
                and for $u,v\in V$, there exist $\alpha_{1},\ldots,
                \alpha_{k}\in\chi(\Gamma)\subset\mathbb{C}^{\times}$ such that
                $$(x_{1}-\alpha_{1}x_{2})\cdots(x_{1}-\alpha_{k}x_{2})[Y_{W}(u,x_{1}), Y_{W}(v,x_{2})] = 0.$$}
\ed

\br{rgradedva-gammava}
{\em Let $V$ be a {\em $\Z$-graded vertex algebra} in the sense that $V$ is a vertex algebra
equipped with a $\Z$-grading $V=\oplus_{n\in \Z}V_{(n)}$ such that ${\bf 1}\in V_{(0)}$ and
$$u_{m}V_{(n)}\subset V_{(k+n-m-1)}\ \ \  \mbox{ for }u\in V_{(k)},\ k,m,n\in \Z.$$
Denote by $L(0)$ the linear operator on $V$, defined by $L(0)|_{V_{(n)}}=n$ for $n\in \Z$.
Define an {\em automorphism} of a $\Z$-graded vertex algebra $V$ to be an automorphism
of vertex algebra $V$, which  preserves the $\Z$-grading.
Let $\Gamma$ be an automorphism group of a $\Z$-graded vertex algebra $V$ and
let $\chi: \Gamma \rightarrow \mathbb{C}^{\times}$ be a group homomorphism.
Then  $V$ becomes a $\Gamma$-vertex algebra
with $R_{g}=\chi(g)^{-L(0)}g$ for $g\in \Gamma$ (see \cite{Li1}).}
\er

Let $\ell$ be a complex number, which is fixed for the rest of this section.
Given the affine Lie algebra $\widehat{\g_{S}}$ and the complex number $\ell$,
we have a $\Z$-graded vertex algebra $V_{\widehat{\g_{S}}}(\ell,0)$ (cf. \cite{FZ}, \cite{LL}).
As a vector space,
   $$ V_{\widehat{\g_{S}}} (\ell,0)=
U(\widehat{\g_{S}})\otimes_{U(\g_{S}\otimes \C[t]+\C {\bf k})}\mathbb{C}_{\ell},   $$
where $\mathbb{C}_{\ell}$ denotes the one-dimensional $(\g_{S}\otimes \C[t]+\C {\bf k})$-module
$\C$ with $\g_{S}\otimes \C[t]$
   acting trivially and with ${\bf k}$ acting as scalar $\ell$, and that $\textbf{1} = 1\otimes 1$,
$Y(d^{\alpha,\beta},x) = d^{\alpha,\beta}(x)$ for $\alpha, \beta\in S$.
Furthermore,
$V_{\widehat{\g_{S}}}(\ell,0)_{(1)}=\g_{S}$
is a generating subspace of $V_{\widehat{\g_{S}}}(\ell,0)$, where
$\g_{S}$ is identified as a subspace of $V_{\widehat{\g_{S}}}(\ell,0)$ through the linear map
$a\mapsto a(-1)\textbf{1}.$

Using a straightforward argument we have:

\begin{lem}\label{automorphism1}
For each $\gamma\in S$, there exists an automorphism $\sigma_{\gamma}$
of the $\Z$-graded vertex algebra $ V_{\widehat{\g_{S}}}(\ell,0)$, which is uniquely determined by
$\sigma_{\gamma}(d^{\alpha,\beta})=d^{\alpha,\beta+\gamma}$ for $\alpha,\beta\in S$.
\end{lem}

Set
\begin{eqnarray}
\Gamma=S\subset
{\mathrm{Aut}}(V_{\widehat{\g_{S}}}(\ell,0)), \label{eq:2.51}
\end{eqnarray}
where $V_{\widehat{\g_{S}}}(\ell,0)$ is viewed as a $\Z$-graded vertex algebra.
Define a group homomorphism
$$R: \  \Gamma \rightarrow \ {\mathrm{GL}}(V_{\widehat{\g_{S}}}(\ell,0))$$
 by
$$ R_{\gamma} =\chi(\gamma)^{-L(0)}\gamma
  \   \  \mbox{for}\;\; \gamma \in S,$$
 where $L(0)$ is the linear operator on $V_{\widehat{\g_{S}}}(\ell,0)$ defined by $L(0)v=nv$
 for $v\in V_{\widehat{\g_{S}}}(\ell,0)_{(n)}$ with $n\in \mathbb{Z}$.
In view of Remark \ref{rgradedva-gammava},  equipped with group homomorphisms $R$ and $\chi$,
the $\Z$-graded vertex algebra $V_{\widehat{\g_{S}}}(\ell,0)$ becomes a $\Gamma$-vertex algebra.

\bd{restricted}
{\em A $D_{S}$-module $W$ is said to be {\em restricted} if
                    for every $\alpha\in S$ and $w\in W,$ $D^{\alpha}(n)w = 0 $
                    for $n$ sufficiently large, or equivalently, if
                    $D^{\alpha}(x)\in \Hom (W,W((x)))$, or equivalently, if
                    $\tilde{D}^{\alpha}(x)\in  \Hom (W,W((x)))$ for all $\alpha\in S$.
We say a $D_{S}$-module W is of {\em level} $\ell\in \C$
                if the central element ${\bf c}$ acts as scalar $\ell.$}
\ed

As the second main result of this section we have:

\begin{thm}\label{quasi-module-1}
Let $W$ be a restricted $D_S$-module of level $\ell$. Then there exists an equivariant quasi
                 $V_{\widehat{\g_{S}}}(\ell,0)$-module
                 structure $Y_{W}(\cdot,x)$ on $W$, which is uniquely determined by
                 $$Y_{W}(d^{\alpha,\beta},x) = D^{\alpha,\beta}(x)
                 =\chi(\beta)\tilde{D}^{\alpha}(\chi(\beta)x)\  \
                 \mbox{ for }\alpha,\beta\in S.$$
On the other hand, let $(W,Y_{W})$ be an equivariant quasi $V_{\widehat{\g_{S}}}(\ell,0)$-module.
Then  $W$ is a restricted $D_{S}$-module of level $\ell$
                  with $\tilde{D}^{\alpha}(x)=Y_{W}(d^{\alpha,0},x)$
                  for $\alpha\in S.$
\end{thm}

\begin{proof} Recall from Theorem \ref{pisomorphism} that $D_{S}$ is isomorphic to
the covariant Lie algebra $\widehat{\g_{S}}/S$ of the affine Lie algebra $\widehat{\g_{S}}$.
Note that $\chi: S\rightarrow \C^{\times}$ is assumed to be injective. Then all the assertions follow immediately from
Theorem 4.9 of \cite{li-gamma-twisted}.
\end{proof}

\section{Identifying Lie algebra $D_{S}$ with affine Kac-Moody algebras}
In this section,  we relate Lie algebra $D_{S}$ to affine Kac-Moody algebras.
To do this, we determine the Lie algebra $\g_{S}$ introduced in Section 2.
Especially, we show that if $S$ is a finite abelian group of order $2l+1$, then
$\g_{S}$ is a finite-dimensional simple Lie algebra of type $B_{l}$ and
Lie algebra $D_{S}$ is isomorphic to the affine Kac-Moody algebra of type $B_{l}^{(1)}$.
We also obtain some results for $S$ of an even order.

Let $S$ be an abelian group as before, which is fixed throughout this section.  Define
$\gl_{S}$ to be the associative algebra with generators $E_{\alpha,\beta}$ for $\alpha,\beta\in S$, subject to relations
\begin{eqnarray}
E_{\alpha,\beta}\cdot E_{\mu,\nu}=\delta_{\beta,\mu}E_{\alpha,\nu}\   \   \  \mbox{ for }\alpha,\beta,\mu,\nu\in S.
\end{eqnarray}
It is straightforward to see that $E_{\alpha,\beta}$ for $\alpha,\beta\in S$ form a basis of $\gl_{S}$.
Notice that if $S$ is of order $n$, then $\gl_{S}$ is isomorphic to the algebra $M(n,\C)$ of  $n\times n$ complex matrices.
Naturally, $\gl_{S}$ is a Lie algebra with
$$[E_{\alpha,\beta}, E_{\mu,\nu}]=E_{\alpha,\beta}\cdot E_{\mu,\nu}-E_{\mu,\nu}\cdot E_{\alpha,\beta}=\delta_{\beta,\mu}E_{\alpha,\nu}-\delta_{\nu,\alpha}E_{\mu,\beta}.$$
Equip $\gl_{S}$ with the bilinear form $\langle \cdot,\cdot\rangle$ defined by
\begin{eqnarray}\label{eBform}
\langle E_{\alpha,\beta},E_{\mu,\nu}\rangle =\frac{1}{2}\delta_{\alpha,\nu}\delta_{\beta,\mu}
\end{eqnarray}
for $\alpha,\beta,\mu,\nu\in S$. It is straightforward to show that this form is symmetric, associative (invariant), and nondegenerate.
%Let $\tau$ be the anti-automorphism of the associative algebra $\gl_{S}$, defined by
Define a linear endomorphism $\tau$ of $\gl_{S}$ by
 \begin{eqnarray}\label{etau-def}
 \tau(E_{\alpha,\beta})=-E_{\beta,\alpha}\ \ \  \mbox{ for }\alpha,\beta\in S.
 \end{eqnarray}
Then $\tau$ is an order-$2$ automorphism of $\gl_{S}$ viewed as a Lie algebra  and we see that $\tau$ preserves the bilinear form.

For $\alpha,\beta\in S$, set
\begin{eqnarray}\label{eDef-G}
G_{\alpha,\beta}=E_{\alpha+\beta,\beta-\alpha}.
\end{eqnarray}
We have
$$G_{\alpha,\beta}\cdot G_{\mu,\nu}=\delta_{\beta-\alpha,\mu+\nu}G_{\alpha+\mu,\alpha+\nu}
\   \   \   \   \mbox{for }\alpha,\beta,\mu,\nu\in S.$$
Furthermore, set
\begin{eqnarray}\label{eq:2.25}
\A_{S}=\mbox{span}\{G_{\alpha,\beta}\   |\  \alpha,\beta\in S\}\subset \gl_{S}.
\end{eqnarray}
By a straightforward argument we get the following results:

\begin{lem}
The subspace $\mathcal{A}_{S}$  of $\gl_{S}$ is an associative subalgebra with a basis
\begin{eqnarray}\label{eAS2S}
\{E_{\mu,\nu}\   |\  \mu,\nu\in S\ \mbox{with}\   \mu+\nu\in 2S\}.
\end{eqnarray}
Furthermore, the automorphism $\tau$ of the Lie algebra $\gl_{S}$
 preserves $\mathcal{A}_{S}$, where
 \begin{eqnarray}
 \tau (G_{\alpha,\beta})=-G_{-\alpha,\beta}\  \  \  \mbox{  for }\alpha,\beta\in S.
 \end{eqnarray}
 \end{lem}

Recall that $S^{0}=\{ \alpha\in S\ |\  2\alpha=0\}$.  We have:

\begin{lem} \label{lGGequality}
For $\alpha,\beta,\alpha',\beta'\in S$, $G_{\alpha,\beta}=G_{\alpha',\beta'}$
if and only if $\alpha=\alpha'+\gamma,\ \beta=\beta'-\gamma=\beta'+\gamma$ for some $\gamma\in S^0$.
\end{lem}

\begin{proof}
It is straightforward as $E_{\mu,\nu}$ for $\mu,\nu\in S$ form a basis of $\gl_{S}$.
\end{proof}

 Set
\begin{eqnarray}
 {\mathcal{A}}_{S}^{\tau}=\{ a\in {\mathcal{A}}_{S}\ |\  \tau(a)=a\},
\end{eqnarray}
a Lie subalgebra of $\mathcal{A}_{S}$.
For $\alpha,\beta\in S$, set
 \begin{eqnarray}\label{eq:2.27}
 G^{\tau}_{\alpha,\beta}=G_{\alpha,\beta}+\tau(G_{\alpha,\beta})
 =G_{\alpha,\beta}-G_{-\alpha,\beta}=E_{\alpha+\beta,\beta-\alpha}-E_{\beta-\alpha,\beta+\alpha}\in {\mathcal{A}}^{\tau}_{S}.
 \end{eqnarray}
 Then
 \begin{eqnarray}
 {\mathcal{A}}_{S}^{\tau}=\span\{ G^{\tau}_{\alpha,\beta}\ |\ \alpha,\beta\in S\}.
 \end{eqnarray}
We have
\begin{eqnarray}\label{esymmetry-Gtau}
G^{\tau}_{\alpha+\gamma,\beta+\gamma}=G^{\tau}_{\alpha,\beta}\  \  \mbox{ for }\alpha,\beta\in S,\ \gamma\in S^{0},
\end{eqnarray}
\begin{eqnarray}\label{esymmetry-G}
G^{\tau}_{-\alpha,\beta}=-G^{\tau}_{\alpha,\beta}\  \  \mbox{ for }\alpha,\beta\in S.
\end{eqnarray}
On the other hand, by a straightforward calculation we get
 \begin{eqnarray}\label{eq:2.28}
  [G^{\tau}_{\alpha,\beta},G^{\tau}_{\mu,\nu}] &=
   & \delta_{\alpha+\mu,\beta-\nu}G^{\tau}_{\alpha+\mu,\nu+\alpha}-\delta_{\alpha+\mu,\nu-\beta}G^{\tau}_{\alpha+\mu,\mu+\beta}
                         \nonumber\\
   &&{}+ \delta_{\alpha-\mu,\nu-\beta}G^{\tau}_{\alpha-\mu,\beta-\mu}- \delta_{\alpha-\mu,\beta-\nu}G^{\tau}_{\alpha-\mu,\alpha+\nu}
\end{eqnarray}
for $\alpha,\beta,\mu,\nu\in S$.

\begin{lem}\label{rho-automorphism}
The group $S$ acts on the Lie algebra ${\mathcal{A}}_{S}^{\tau}$
as  an automorphism group  with $\gamma\in S$ acting as $\sigma_{\gamma}$, where
\begin{eqnarray}
\sigma_{\gamma}(G_{\alpha,\beta}^{\tau})=G_{\alpha,\beta+\gamma}^{\tau}
\end{eqnarray}
for $\alpha,\beta\in S$. Furthermore, the action preserves the bilinear form $\<\cdot,\cdot\>$ defined in (\ref{eBform}).
\end{lem}

\begin{proof} First, $S$ acts on $\gl_{S}$ as an automorphism group by
\begin{eqnarray}\label{new}
\sigma_{\gamma}(E_{\alpha,\beta})=E_{\alpha+\gamma, \beta+\gamma} \   \   \   \mbox{ for }\gamma,\alpha,\beta\in S.
\end{eqnarray}
It is clear that $S$ preserves the bilinear form $\langle\cdot,\cdot\rangle$ on $\gl_{S}$.
Second, for $\alpha,\beta,\gamma\in S$, we have
$$\sigma_{\gamma}(G_{\alpha,\beta})=E_{\alpha+\beta+\gamma,\beta-\alpha+\gamma}=G_{\alpha,\beta+\gamma}$$
 and furthermore, we have
$$\sigma_{\gamma}(G_{\alpha,\beta}^{\tau})=G_{\alpha,\beta+\gamma}-G_{-\alpha,\beta+\gamma}
=G_{\alpha,\beta+\gamma}^{\tau}.$$
Consequently, $S$ acts on ${\mathcal{A}}_{S}^{\tau}$ as an automorphism group.
\end{proof}

Next, we relate the Lie algebra $\g_{S}$ introduced in Section 2 with Lie subalgebra ${\mathcal{A}}_{S}^{\tau}$
of $(\gl_{S})^{\tau}$.
Recall that $S$ also acts on $\g_{S}$ as an automorphism group.
From  relations (\ref{esymmetry-G}) and (\ref{eq:2.28}) and  the definition of $\g_{S}$
we immediately have:

\bl{homomorphism}
There exists a Lie algebra homomorphism $\pi$ from $\g_{S}$
 onto ${\mathcal{A}}_{S}^{\tau}\subset (\gl_{S})^{\tau}$ such that
\begin{eqnarray}
\pi(d^{\alpha,\beta})=G_{-\alpha,\beta}^{\tau}=-G^{\tau}_{\alpha,\beta}=E_{\beta-\alpha,\alpha+\beta}-E_{\alpha+\beta,\beta-\alpha}
\   \   \   \   \mbox{ for }\alpha,\beta\in S.
\end{eqnarray}
Furthermore,  $\pi$ preserves the group actions of $S$.
\el

Furthermore, we have:

\begin{prop} \label{realization}
Assume that $S$ is an abelian group such that $S^{0}=\{0\}$.
Then the Lie algebra homomorphism $\pi$ defined in Lemma \ref{homomorphism} from $\g_{S}$ to ${\mathcal{A}}_{S}^{\tau}$  is an isomorphism, which also preserves the invariant bilinear form.
\end{prop}

\begin{proof} For $\alpha,\beta,\mu,\nu\in S$, we have %(with the factor $1/2$)
\begin{eqnarray}
\< G_{\alpha,\beta}^{\tau},G^{\tau}_{\mu,\nu}\>&=&\<E_{\alpha+\beta,\beta-\alpha}-E_{\beta-\alpha,\alpha+\beta},
E_{\mu+\nu,\nu-\mu}-E_{\nu-\mu,\mu+\nu}\>\nonumber\\
&=&\delta_{\alpha+\beta,\nu-\mu}\delta_{\beta-\alpha,\mu+\nu}
-\delta_{\alpha+\beta,\mu+\nu}\delta_{\beta-\alpha,\nu-\mu}\nonumber\\
&=&\delta_{2(\alpha+\mu),0}\delta_{\beta-\alpha,\mu+\nu}
-\delta_{2(\alpha-\mu),0}\delta_{\beta-\alpha,\nu-\mu}\nonumber\\
&=&\delta_{\alpha+\mu,0}\delta_{\beta-\alpha,\mu+\nu}
-\delta_{\alpha-\mu,0}\delta_{\beta-\alpha,\nu-\mu}\nonumber\\
&=&\left(\delta_{\alpha,-\mu}-\delta_{\alpha,\mu}\right)\delta_{\beta,\nu}.
\end{eqnarray}
It follows from this and (\ref{bilinear-form-gs}) that $\pi$ preserves the invariant bilinear form.
For the isomorphism assertion it remains to show that $\pi$ is injective.
As $S^{0}=\{0\}$, in view of Lemma \ref{lGGequality}, we have that
$G_{\alpha,\beta}=G_{\alpha',\beta'}$
if and only if $\alpha=\alpha'$ and $\beta=\beta'$ for $\alpha,\beta,\alpha',\beta'\in S$. Consequently,
$G_{\alpha,\beta}$  for $\alpha,\beta\in S$ form a basis of $\mathcal{A}_{S}$.
Let $S_-$ be a subset of $S$ such that $S=\{0\}\cup S_-\cup (-S_{-})$ is a disjoint decomposition.
As $G^{\tau}_{-\alpha,\beta}=-G^{\tau}_{\alpha,\beta}$ for $\alpha,\beta\in S$, $\mathcal{A}_{S}^{\tau}$ is linearly spanned by
$G_{\alpha,\beta}^{\tau}$  for $\alpha\in S_{-},\  \beta\in S$, which are linearly independent since $G_{\alpha,\beta}$  for $\alpha,\beta\in S$ are linearly independent.
Then $G_{\alpha,\beta}^{\tau}$  for $\alpha\in S_{-},\  \beta\in S$ form a basis of $\mathcal{A}_{S}^{\tau}$.
On the other hand, with (\ref{ed-symmetry}), we see that
$d^{\alpha,\beta}$ with $\alpha\in S_-,\ \beta\in S$ linearly span $\g_{S}$. As $\pi$ is a linear map from $\g_{S}$ onto $\mathcal{A}_{S}^{\tau}$ with $\pi (d^{\alpha,\beta})=-G_{\alpha,\beta}^{\tau}$ for $\alpha,\beta\in S$,
elements $d^{\alpha,\beta}$ for $\alpha\in S_-,\ \beta\in S$
must be linearly independent, so that $\{ d^{\alpha,\beta}\ |\ \alpha\in S_-,\ \beta\in S\}$
is a basis of $\g_{S}$. Consequently, $\pi$ is injective. This completes the proof.
\end{proof}

As a consequence of Proposition \ref{realization}, we have:

\bc{isomorphism}
Assume that $S$ is a finite  abelian group of order $2l+1$ with $l$ a positive integer.
Then ${\mathcal{A}}_{S}=\gl_{S}$ and $\g_{S}$ is isomorphic to $(\gl_{S})^{\tau}$
which is isomorphic to the simple Lie algebra of type $B_{l}$.
 \ec

\begin{proof}  As $|S|$ is odd, there exist integers $m$ and $n$ such that $1=2m+n|S|$.
Then $\alpha=2m\alpha$ for $\alpha\in S$. It follows from (\ref{eDef-G}) that for any $\alpha,\beta\in S$,
\begin{eqnarray}
E_{\alpha,\beta}=G_{m(\alpha-\beta),m(\alpha+\beta)}\in {\mathcal{A}}_{S}.
\end{eqnarray}
Thus ${\mathcal{A}}_{S}=\gl_{S}$ and hence ${\mathcal{A}}_{S}^{\tau}=(\gl_{S})^{\tau}$.
Again, since $|S|$ is odd, we have $S^{0}=\{0\}$.
By Proposition \ref{realization}, $\g_{S}\simeq {\mathcal{A}}_{S}^{\tau}=(\gl_{S})^{\tau}$.
It is well known that $(\gl_{S})^{\tau}$ is isomorphic to the finite-dimensional simple Lie algebra of type $B_{l}$.
%(the correspondence between the defined bilinear forma and the normalized killing form???)
This completes the proof.
\end{proof}

Furthermore, we have:

\bt{identification}
Assume that $S$ is a finite  abelian group of order $2l+1$ with $l$ a positive integer.
Then Lie algebra $D_{S}$ is isomorphic to the affine Kac-Moody algebra of type $B_{l}^{(1)}$.
%with the central element ${\bf c}$ of $D_{S}$ corresponding to the central element ${\bf k}$ of $B_{l}^{(1)}$.
\et

\begin{proof} Note that as $\chi: S\rightarrow \C^{\times}$ is assumed to be one-to-one,  $S$ must be cyclic.
Assume $S=\Z/(2l+1)\Z$ and let $\chi: S\rightarrow \C^{\times}$ be the group homomorphism defined by
$\chi(\alpha)=q^{\alpha}$ for $\alpha\in S$, where $q$ is a primitive $(2l+1)$-th root of unity.
From Theorem \ref{pisomorphism} and Corollary \ref{isomorphism},  we have
\begin{eqnarray}\label{3.18}
D_{S}\simeq \widehat{\g_S}/S,   \     \    \   \g_S\simeq  (\gl_{S})^{\tau}\simeq B_{l}.
\end{eqnarray}
Let $\sigma_1$ be the automorphism of the Lie algebra $(\gl_{S})^{\tau}\ (\simeq B_{l})$, given by
\begin{eqnarray}
\sigma_1(E_{\alpha,\beta}-E_{\beta,\alpha})=E_{\alpha+[1],\beta+[1]}-E_{\beta+[1],\alpha+[1]}
\end{eqnarray}
for $\alpha,\beta\in \Z/(2l+1)\Z=S$, where $[1]$ denotes the image of $1$ in $ \Z/(2l+1)\Z$.
With $S$ viewed as an automorphism group of
$(\gl_{S})^{\tau}$  (recalling (\ref{new})), we have $S=\< \sigma_1\>$.
Then, using (\ref{3.18}) we get
$$D_{S}\simeq \widehat{B_{l}}/S\simeq \widehat{B_{l}}^{S}= \widehat{B_{l}}^{\sigma_1},
%\simeq \hat{B_{l}}[\sigma_1],
$$
where $\widehat{B_{l}}^{S}$ (resp. $\widehat{B_{l}}^{\sigma_1}$) denotes the $S$-fixed
(resp. $\sigma_1$-fixed) points subalgebra. Note that for the isomorphism relation
$\widehat{B_{l}}/S\simeq \widehat{B_{l}}^{S}$, ${\bf k}$ corresponds to $(2l+1){\bf k}$.

Note that there is no nontrivial Dynkin diagram automorphism for a type $B_{l}$ simple Lie algebra.
From \cite{kac} (Proposition 8.1), $\sigma_1$ is conjugate to an inner automorphism
$\exp \left(\frac{2\pi i}{2l+1}\ad h\right)$, where $h$ is an element such that $\ad h$ is semi-simple
with only integer eigenvalues.
It follows that $\widehat{B_{l}}^{\sigma_1}$ is isomorphic to $\widehat{B_{l}}$ (see Appendix; cf. Section 8.5, \cite{kac}).
Consequently, $D_{S}$ is isomorphic to the affine Kac-Moody algebra of type $B_{l}^{(1)}$.
%with ${\bf c}$ corresponding to ${\bf k}$.
\end{proof}

Next, we consider the general case.
Recall ${\mathcal{A}}_{S}={\rm span}\{ E_{\alpha+\beta,\alpha-\beta}\ |\ \alpha,\beta\in S\}$,
which is an associative subalgebra  of $\gl_{S}$, and
$${\mathcal{A}}_{S}^{\tau}={\rm span}\{ G^{\tau}_{\alpha,\beta}
=E_{\beta+\alpha,\beta-\alpha}-E_{\beta-\alpha,\beta+\alpha}\ |\ \alpha,\beta\in S\},$$
a Lie subalgebra of ${\mathcal{A}}_{S}$. We also
recall $S^0=\{ \alpha\in S\ |\ 2\alpha=0\}$. The following is a generalization of Proposition \ref{realization}:

\begin{prop}\label{ideal-I}
Set
\begin{eqnarray}
I={\rm span}\{ d^{\alpha+\gamma,\beta+\gamma}-d^{\alpha,\beta}\ |\ \alpha,\beta\in S,\ \gamma \in S^0\}\subset \g_{S}.
\end{eqnarray}
Then $I$ is an ideal of $\g_{S}$.
Furthermore, the quotient Lie algebra $\g_{S}/I$ is isomorphic to ${\mathcal{A}}_{S}^{\tau}$.
\end{prop}

\begin{proof}
From (\ref{eq:2.24}), it is straightforward to show that $I$ is an ideal.
On the other hand,  by Lemma \ref{lGGequality}, for $\alpha,\beta\in S,\ \gamma \in S^0$, we have
$$G_{\alpha+\gamma,\beta+\gamma}^{\tau}=G_{\alpha+\gamma,\beta+\gamma}-G_{-\alpha-\gamma,\beta+\gamma}
=G_{\alpha,\beta}-G_{-\alpha,\beta}
=G_{\alpha,\beta}^{\tau}$$
in ${\mathcal{A}}_{S}^{\tau}$. Then the Lie algebra homomorphism $\pi$ from $\g_{S}$ onto ${\mathcal{A}}_{S}^{\tau}$
 reduces to a Lie algebra homomorphism $\bar{\pi}$ from $\g_{S}/I$ onto ${\mathcal{A}}_{S}^{\tau}$ with
$\bar{\pi}(d^{\alpha,\beta}+I)=-G_{\alpha,\beta}^{\tau}$ for $\alpha,\beta\in S$.
It suffices to prove that $\bar{\pi}$ is injective.
Let $T$ be a complete set of equivalence class representatives of $S$ modulo $S^0$ with $0\in T$.
It follows from Lemma \ref{lGGequality} that $\{G_{\alpha,\beta}\ | (\alpha,\beta)\in T\times  S\}$ is a base for ${\mathcal{A}}_{S}$.
Furthermore, consider the equivalence relation $\sim$ on $S$ defined by $\alpha\sim \beta$ if and only if $\beta=\pm \alpha$.
Let $T_{-}$ be a complete set of equivalence class representatives in $T\setminus \{0\}$ with respect to this particular
equivalence relation. Noticing that if $\alpha\in T_{-}$, then $-\alpha\notin T_{-}$, we see
that the linear independence of $\{G_{\alpha,\beta}\ | (\alpha,\beta)\in T\times  S\}$
implies that $G_{\alpha,\beta}^{\tau}$ with $\alpha\in T_{-},\ \beta\in S$ are linearly independent. Then
 $\{G_{\alpha,\beta}^{\tau}\ |\  (\alpha,\beta)\in T_{-}\times S\}$ is a base for ${\mathcal{A}}_{S}^{\tau}$.
 On the other hand, from the defining relations of $\g_{S}$ and the definition of $I$,
 we see that $\g_{S}/I$ is linearly spanned by $d^{\alpha,\beta}+I$ with $(\alpha,\beta)\in T_{-}\times S$.
It follows that $\{d^{\alpha,\beta}+I\ |\  (\alpha,\beta)\in T_{-}\times S\}$ is a base for $\g_{S}/I$.
Consequently, $\bar{\pi}$ is a Lie algebra isomorphism.
\end{proof}

Recall the subgroup $2S$ of $S$. Let
\begin{eqnarray}
S=S_{0}\cup S_1\cup\cdots \cup S_{r-1}
\end{eqnarray}
be the decomposition of $S$ into distinct cosets of $2S$ with $S_0=2S$.
 We have $2S\simeq S/S^0$, which implies $|2S|\cdot |S^0|=|S|$.
Thus $r=|S^0|$.
 Notice that for $\alpha,\beta\in S$, if $\alpha+\beta\in S_{j}$ for $0\le j<r$, we have
 $\beta-\alpha=(\alpha+\beta)-2\alpha\in S_j$, so that
 $$G^{\tau}_{\alpha,\beta}=E_{\beta+\alpha,\beta-\alpha}-E_{\beta-\alpha,\beta+\alpha}\in \gl_{S_{j}}.$$
 Thus
\begin{eqnarray}
{\mathcal{A}}_{S}^{\tau}\subset
\gl_{S_0}\oplus \gl_{S_1}\oplus \cdots \oplus \gl_{S_{r-1}}\subset \gl_{S}.
\end{eqnarray}
For $0\le j\le r-1$, set
$$\g_j={\mathcal{A}}_{S}^{\tau}\cap \gl_{S_j},$$
a Lie subalgebra of ${\mathcal{A}}_{S}^{\tau}$. Then
\begin{eqnarray}
{\mathcal{A}}_{S}^{\tau}=\g_{0}\oplus \g_{1}\oplus \cdots \oplus \g_{r-1},
\end{eqnarray}
a direct sum of Lie algebras.

Furthermore, we have:

\begin{lem}\label{lgj-exp}
For each $0\le j<r$, we have
\begin{eqnarray}\label{gj-expression}
\g_{j}={\rm span}\{ E_{\mu,\nu}-E_{\nu,\mu}\ | \  \mu,\nu\in S_j\}
\end{eqnarray}
and $\dim \g_j=\frac{1}{2}k(k-1)$ where $k=|2S|$.
\end{lem}

\begin{proof}  From definition we have
\begin{eqnarray*}
\g_j={\rm span}\{
E_{\beta+\alpha,\beta-\alpha}-E_{\beta-\alpha,\beta+\alpha}\ |\ \alpha,\beta\in S\  \mbox{ with }\alpha+\beta\in S_j\},
\end{eqnarray*}
noticing that $\alpha+\beta\in S_j$ implies $\beta-\alpha=(\alpha+\beta)-2\alpha\in S_j$. It follows that
$\g_{j}\subset {\rm span}\{ E_{\mu,\nu}-E_{\nu,\mu}\ | \  \mu,\nu\in S_j\}$.
On the other hand, let $\mu,\nu\in S_{j}$. Fix an element  $\alpha_j\in S_j$. Then
$\mu=\alpha_j+2\alpha$ and $\nu=\alpha_j+2\beta$
for some $\alpha,\beta\in S$. Thus
$$E_{\mu,\nu}-E_{\nu,\mu}=G_{\alpha-\beta,\alpha+\beta+\alpha_j}^{\tau}\in \g_j.$$
This proves that $\g_{j}\supset {\rm span}\{ E_{\mu,\nu}-E_{\nu,\mu}\ | \  \mu,\nu\in S_j\}$. Hence (\ref{gj-expression}) holds.
From this the second assertion of the lemma follows immediately.
\end{proof}

Note that $\gl_{S}$ naturally acts on the group algebra $\C[S]$ of $S$.
Equip $\C[S]$ with a nondegenerate symmetric bilinear form
$\<\cdot,\cdot\>$ defined by $\<e^{\alpha},e^{\beta}\>=\delta_{\alpha,\beta}$ for $\alpha,\beta\in S$.
As subalgebras of $\gl_{S}$,
${\mathcal{A}}_{S}$ and ${\mathcal{A}}_{S}^{\tau}$ also act on $\C[S]$.
We have the following ${\mathcal{A}}_{S}^{\tau}$-module decomposition
$$\C[S]=\C[S_0]\oplus \C[S_1]\oplus \cdots \oplus \C[S_{r-1}],$$
where $\g_{j}$ acts on $\C[S_j]$ for $0\le j<r$. Denoting by $\rho$ the algebra homomorphism of $\gl_{S}$ to
$\gl(\C[S])$, we have
$$\rho(\g_j)\subset \{ A\in \gl(\C[S_j])\ | \  (Au,v)=-(u,Av)\  \  \mbox{ for }u,v\in \C[S_j]\}.$$
Since $|S_j|=k$ and $\dim \g_j=\frac{1}{2}k(k-1)$ by Lemma \ref{lgj-exp}, we see that actually the equality holds.
Therefore,  $\g_j$ is a simple Lie algebra of type either $B$ or $D$, depending on
that $k$ ($=|2S|$) is odd or even.  Consequently, ${\mathcal{A}}_{S}^{\tau}$ is semi-simple.

To summarize, we have:

\begin{prop}\label{general-case}
Let $S$ be a finite  abelian group. Set $k=|2S|$ and $r=|S/2S|$.
Then $\g_{S}/I$ is isomorphic to the direct sum of $r$ copies of simple Lie algebras
of type $B_{\frac{1}{2}(k-1)}$ if $k$ is odd and of type $D_{\frac{1}{2}k}$ if $k$ is even,
where $I$ is the ideal of $\g_{S}$, which was obtained in Proposition \ref{ideal-I}.
\end{prop}

\section*{Appendix}
Here, we provide a proof of a statement at the end of the proof of Theorem \ref{identification}.

\bl{appendix}
 Suppose that $\g$ is a Lie algebra with a symmetric invariant bilinear form $\<\cdot,\cdot\>$ and
$h$ is an element of $\g$ such that $\ad h$ is semi-simple with only integer eigenvalues.
Set $\sigma=\exp (\frac{2\pi i}{T}\ad h)$, where $T$ is a fixed positive integer.
Then $\hat{\g}^{\sigma}\simeq \hat{\g}$
with ${\bf k}$ corresponding to $\frac{1}{T}{\bf k}$.
\el

\begin{proof}  Recall the affine Lie algebra $\hat{\g}=\g\otimes \C[t,t^{-1}]\oplus \C {\bf k}$, where
$$[a\otimes t^{m},b\otimes t^{n}]=[a,b]\otimes t^{m+n}+m\delta_{m+n,0}\<a,b\>{\bf k}$$
for $a,b\in \g,\ m,n\in \Z$. On the other hand, we have  Lie algebra
$$\widehat{L}(\g,T)=\g\otimes \C[t^{\frac{1}{T}},t^{-\frac{1}{T}}]\oplus \C {\bf k},$$
where
$$[a\otimes t^{\frac{m}{T}},b\otimes t^{\frac{n}{T}}]=[a,b]\otimes t^{\frac{m+n}{T}}+\frac{m}{T}\delta_{m+n,0}\<a,b\>{\bf k}$$
for $a,b\in \g,\ m,n\in \Z$.
It is straightforward to see that  $\hat{\g}$ is canonically isomorphic to $\widehat{L}({\g},T)$
with ${\bf k}$ corresponding to $\frac{1}{T}{\bf k}$.
 View $\sigma$ as an automorphism of $\widehat{L}({\g},T)$ by
 $$\sigma(a\otimes t^{\frac{n}{T}})=\exp\left(\frac{2n\pi i}{T}\right)\sigma(a)\otimes t^{\frac{n}{T}}\   \   \
\mbox{ for }a\in \g,\ n\in \Z,$$ and then set
$$\hat{\g}[\sigma]=\widehat{L}(\g,T)^{\sigma}.$$
Then $\hat{\g}^{\sigma}$ is canonically isomorphic to $\hat{\g}[\sigma]$
with ${\bf k}$ corresponding to $\frac{1}{T}{\bf k}$.
Now, it suffices to prove that $\hat{\g}[\sigma]$ is isomorphic to $\hat{\g}$ with
${\bf k}$ corresponding to ${\bf k}$.
For $r\in \Z$, set $\g_{r}=\{ u\in \g\ |\ [h,u]=ru\}$. Then $\g=\oplus_{r\in \Z}\g_{r}$ and
$\<\g_r,\g_s\>=0$ whenever $r+s\ne 0$.
For $a\in \g_{r}$ with $r\in \Z$, set
$$a^{\sigma}(x)=\sum_{n\in \Z}(a\otimes t^{n-\frac{r}{T}})x^{-n-1+\frac{r}{T}}\in x^{\frac{r}{T}}\hat{\g}[\sigma][[x,x^{-1}]].$$
 It was esentially proved in  \cite{Li-twisted-local} (Remark 5.5) that
the linear map $\psi: \hat{\g}[\sigma]\rightarrow \hat{\g}$ defined by  $\psi({\bf k})={\bf k}$ and
$$\psi( a^{\sigma}(x))=x^{\frac{r}{T}}\left(a(x)+\frac{1}{T}\<h,a\>{\bf k}x^{-1}\right)$$
%=x^{\frac{r}{T}}a(x)+\frac{1}{T}\<h,a\>{\bf k}x^{-1}$$
for $a\in \g_{r}$ with $r\in \Z$, is a Lie algebra isomorphism.
In terms of components, we have
 $$\psi( a\otimes t^{n-\frac{r}{T}})=a\otimes t^{n}+\frac{1}{T}\<h,a\> \delta_{n,0}{\bf k},$$
 noticing that $\<h,a\>=\delta_{r,0}\<h,a\>$ as $h\in \g_0$.
 On the other hand, it is straightforward to show that
$\psi$ is indeed a Lie algebra isomorphism.
\end{proof}

\end{document}